\documentclass[a4paper,oneside]{scrartcl}

\usepackage[english]{babel}
\usepackage{microtype}
\usepackage{amsfonts}
\usepackage{amssymb}
\usepackage{amsmath}
\usepackage{amsthm}
\usepackage{verbatim}
\usepackage{algorithmic}
\usepackage{graphicx}
\usepackage{color}
\usepackage{wrapfig}

\theoremstyle{plain}
\newtheorem{theorem}{Theorem}[section]
\newtheorem{lemma}[theorem]{Lemma}
\newtheorem{proposition}[theorem]{Proposition}
\newtheorem{corollary}[theorem]{Corollary}

\newtheorem{definition}[theorem]{Definition}

\newcommand{\rp}{(\mathbb{R}^n,\beta)}
\newcommand{\rpm}{\mathbb{R}^n,\beta}
\newcommand{\rsq}{\mathbb{R}^n}

\newcommand{\mf}{\mathfrak}
\newcommand{\mi}{\mathit}

\newcommand{\Title}{Undecidable First-Order Theories of Affine Geometries}
\newcommand{\Author}{
Antti Kuusisto%
\thanks{University of Tampere, \{antti.j.kuusisto, jonni.virtema\}@uta.fi}
, Jeremy Meyers%
\thanks{Stanford University, jjmeyers@stanford.edu}
, Jonni Virtema%
\footnotemark[2]
}

\begin{document}

\title{\Title\thanks{This work was partially supported by grant 129761 of the Academy of Finland.}}
\author{\Author}

\maketitle

\begin{abstract}
Tarski initiated a logic-based approach to formal geometry that studies
first-order structures with a ternary \emph{betweenness} relation ($\beta$) and a quaternary
\emph{equidistance} relation $(\equiv)$. Tarski established, inter alia, that
the first-order ($\mathrm{FO}$) theory of $(\mathbb{R}^2,\beta,\equiv)$ is decidable.
Aiello and van Benthem (2002) conjectured that the $\mathrm{FO}$-theory of expansions of $(\mathbb{R}^2,\beta)$ with unary predicates
is decidable. We refute this conjecture by showing that for all $n\geq 2$, the $\mathrm{FO}$-theory of monadic expansions of $\rp$ is $\Pi_1^1$-hard
and therefore not even arithmetical.
We also define a natural and comprehensive class $\mathcal{C}$ of geometric structures $(T,\beta)$, where $T\subseteq \rsq$, and show that
for each structure $(T,\beta)\in\mathcal{C}$, the $\mathrm{FO}$-theory of the class of monadic expansions of $(T,\beta)$ is undecidable.
We then consider classes of expansions of structures $(T,\beta)$ with restricted unary 
predicates, for example finite predicates, 
and establish a variety of related undecidability results.
In addition to decidability questions, we
briefly study the expressivity of universal $\mathrm{MSO}$ and weak universal $\mathrm{MSO}$ over
expansions of $(\mathbb{R}^n,\beta)$.
While the logics are incomparable in general, over expansions of $(\mathbb{R}^n,\beta)$, formulae of weak universal $\mathrm{MSO}$
translate into equivalent formulae of universal $\mathrm{MSO}$.
\end{abstract}

\section{Introduction}
Decidability of theories of (classes of) structures is a
central topic in various different fields of computer science and
mathematics, with different motivations and objectives
depending on the field in question.
In this article we investigate formal theories of \emph{geometry} in the
framework introduced by Tarski \cite{Tarski:1948, TarskiGivant:1999}.
The logic-based framework was originally presented in a series of lectures 
given in Warsaw in the 1920's.
The system is based on first-order structures with two
predicates: a ternary \emph{betweenness} relation $\beta$ and a quaternary
\emph{equidistance} relation $\equiv$.
Within this system, $\beta(u,v,w)$ is interpreted to mean that the point
$v$ is between the points $u$ and $w$,
while $x y\equiv u v$ means that the
distance from $x$ to $y$ is equal to the distance from $u$ to $v$.
The betweenness relation $\beta$
can be considered to simulate the action of a ruler,
while the equidistance relation $\equiv$ simulates the action of a compass.
See \cite{TarskiGivant:1999} for information about the
history and development of Tarski's geometry. 

Tarski established in \cite{Tarski:1948}
that the first-order theory of
$(\mathbb{R}^2,\beta,\equiv)$ is decidable.
In \cite{vanBenthemAiello:2002}, Aiello and van Benthem
pose the question: ``\emph{What is the complete monadic $\Pi^1_1$ theory of the affine real plane}?''
By \emph{affine real plane}, the authors refer to the structure $(\mathbb{R}^2,\beta)$.
The monadic $\Pi_1^1$-theory of $(\mathbb{R}^2,\beta)$ is of course
essentially the same as the first-order theory of the class of
expansions $(\mathbb{R}^2,\beta,(P_i)_{i\in\mathbb{N}})$
of the the affine real plane $(\mathbb{R}^2,\beta)$ by unary predicates $P_i\subseteq\mathbb{R}^2$.
Aiello and van Benthem conjecture that the theory is decidable. Expansions of $(\mathbb{R}^2,\beta)$
with \emph{unary} predicates are especially relevant in
investigations related to the geometric structure $(\mathbb{R}^2,\beta)$, since in this 
context unary predicates correspond to 
\emph{regions} of the plane $\mathbb{R}^2$.

In this article we study structures of the type
of $(T,\beta)$, where $T\subseteq\mathbb{R}^n$ and $\beta$ is the
canonical Euclidean betweenness predicate restricted to $T$, see
Section \ref{structures} for the
formal definition. Let $E\bigl((T,\beta)\bigr)$
denote the class of expansions $(T,\beta,(P_i)_{i\in\mathbb{N}})$
of $(T,\beta)$ with unary predicates.
We identify a significant collection of canonical structures $(T,\beta)$
with an undecidable first-order theory of $E\bigl((T,\beta)\bigr)$.
Informally, if there exists a flat two-dimensional
region $R\subseteq\mathbb{R}^n$, no matter how small,
such that $T\cap R$ is in a certain sense sufficiently dense with respect to $R$, then the
first-order theory of the class $E\bigl((T,\beta)\bigr)$ is undecidable.
If the related density conditions are satisfied, we say that \emph{$T$ extends
linearly in $2D$}, see Section \ref{structures} for the formal definition.
We prove that for any $T\subseteq\mathbb{R}^n$, if $T$ extends linearly in $2D$,
then the $\mathrm{FO}$-theory of $E\bigl((T,\beta)\bigr)$ is $\Sigma^0_1$-hard.
In addition, we establish that for all $n\geq 2$,
the first-order theory of $E\bigl((\mathbb{R}^n,\beta)\bigr)$ is $\Pi_1^1$-hard,
and therefore not even arithmetical.
We thereby refute the conjecture
of Aiello and van Benthem from \cite{vanBenthemAiello:2002}.
The results are ultimately based on tiling arguments. The result establishing $\Pi_1^1$-hardness
relies on the \emph{recurrent tiling problem} of Harel \cite{Harel:1985}---once
again demonstrating the usefulness of Harel's methods.
Our results establish undecidability for a wide range of monadic expansion classes of natural geometric
structures $(T,\beta)$.
In addition to  $(\mathbb{R}^2,\beta)$, such
structures include for example the rational plane $(\mathbb{Q}^2,\beta)$,
the real unit cube $([0,1]^3,\beta)$, and the plane of
algebraic reals $(\mathbb{A}^2,\beta)$ --- to name a few.
In addition to investigating monadic expansion classes of the type $E\bigl((T,\beta)\bigr)$, we also
study classes of expansions with \emph{restricted} unary predicates. Let $n$ be a positive integer and let $T\subseteq \rsq$.
Let $F\bigl((T,\beta)\bigr)$ denote the class of structures $(T,\beta,(P_i)_{i\in\mathbb{N}})$, where the sets $P_i$ are
\emph{finite} subsets of $T$. We establish that 
if $T$ extends linearly in $2D$, then the first-order theory of $F\bigl((T,\beta)\bigr)$
is undecidable. An alternative
reading of this result is that the \emph{weak} universal monadic second-order theory of $(T,\beta)$ is undecidable.
We obtain a $\Pi_1^0$-hardness result by an argument based on the
\emph{periodic torus tiling problem} of Gurevich and Koryakov \cite{GurevichKoryakov:1972}.
The torus tiling argument  
can easily be adapted to deal with various different kinds of 
natural classes of expansions of geometric structures $(T,\beta)$ with restricted unary
predicates.
%
%
These include the classes with
unary predicates denoting---for example---polygons,
finite unions of closed rectangles, and
real algebraic sets (see \cite{realalgebraicgeometry} for the definition).
Our results could turn out useful in investigations concerning logical
aspects of spatial databases. It turns out that there is a
canonical correspondence between $(\mathbb{R}^2,\beta)$ and
$(\mathbb{R},0,1,\cdot,+,<)$,
see \cite{Gyssens:1999}. 
See the survey $\cite{KujpersVandenBussche}$ for further details on logical
aspects of spatial databases.
The betweenness predicate is also studied in spatial logic \cite{handbookofspatial}. The recent years 
have witnessed a significant increase in the research on spatially motivated logics.
Several interesting systems
with varying motivations have been investigated, see for
example the articles
\cite{vanBenthemAiello:2002, Balbianietal:1997, Balbiani and Goranko:2002,
HodkinsonHussain, KoPrWoZa:2010, Nenov:2010, Shere:2010,Tinchev,Venema:1999}.
See also the surveys \cite{AielloPrattHartmannvanBenthem:2007} and \cite{Balbianietal:2007} in the
Handbook of Spatial Logics
\cite{handbookofspatial}, and the Ph.D. thesis \cite{Griffiths:2008}.
Several of the above articles investigate
fragments of first-order
theories by way of modal logics for
affine, projective, and metric geometries. Our results contribute to the 
understanding of spatially motivated first-order languages, and hence they can be useful in the
search for decidable (modal) spatial logics.
In addition to studying issues of decidability, we briefly 
compare the expressivities of universal monadic second-order logic $\forall\mathrm{MSO}$ 
and weak universal monadic second-order logic $\forall\mathrm{WMSO}$.
It is straightforward to observe that in general,
the expressivities of $\forall\mathrm{MSO}$ and $\forall\mathrm{WMSO}$ are incomparable in a rather 
strong sense: $\forall\mathrm{MSO}\not\leq\mathrm{WMSO}$ and $\forall\mathrm{WMSO}\not\leq\mathrm{MSO}$.
Here $\mathrm{MSO}$ and $\mathrm{WMSO}$ denote monadic second-order logic and weak monadic second-order logic,
respectively. The result $\forall\mathrm{WMSO}\not\leq\mathrm{MSO}$ follows from already existing
results (see \cite{tenCate:2011} for example), and the result $\forall\mathrm{MSO}\not\leq\mathrm{WMSO}$ is
more or less trivial to prove.
While $\forall\mathrm{MSO}$ and $\forall\mathrm{WMSO}$ are
incomparable in general,
the situation changes when we consider
expansions $(\rpm,(R_i)_{i\in I})$ of the stucture $\rp$, i.e., structures
embedded in the geometric structure $\rp$. Here $(R_i)_{i\in I}$ is an arbitrary vocabulary and $I$ an
arbitrary related index set. We show that over such structures,
sentences of $\forall\mathrm{WMSO}$ translate into equivalent
sentences of $\forall\mathrm{MSO}$. The proof is based on the Heine-Borel theorem.
The structure of the current article is as follows. In Section \ref{section2}
we define the central notions
needed in the later sections. In Section \ref{section3} we compare the
expressivities of $\forall\mathrm{MSO}$ and $\forall\mathrm{WMSO}$. 
In Section \ref{main} we show undecidability of the first-order theory of the class of monadic expansions of
any geometric structure $(T,\beta)$ such that $T$ exends linearly in $2D$. In addition, we show
that for $n\geq 2$, the first-order theory of monadic expansions of $\rp$ is
not on any level of the arithmetical hierarchy. In Section \ref{section5} we
modify the approach in Section \ref{main} and
show undecidability of the $\mathrm{FO}$-theory of
the class of expansions by finite unary predicates of any
geometric structure $(T,\beta)$ such that $T$ extends linearly in $2D$.

\section{Preliminaries}\label{section2}
\subsection{Interpretations}
Let $\sigma$ and $\tau$ be relational vocabularies. Let $\mathcal{A}$ be a
nonempty class of $\sigma$-structures and  $\mathcal{C}$ a nonempty class of $\tau$-structures.
Assume that there exists a surjective map $F$ from $\mathcal{C}$ onto $\mathcal{A}$ and a
first-order $\tau$-formula $\varphi_{\mi{Dom}}(x)$ in one free variable, $x$,
such that for each structure $\mf{B}\in\mathcal{C}$, there is a bijection $f$ from the domain of $F(\mf{B})$ to the set
$$\{\ b\in \mi{Dom}(\mf{B})\ |\ \mf{B}\models\varphi_{\mi{Dom}}(b)\ \}.$$
Assume, furthermore, that for each relation symbol $R\in\sigma$, there is a first-order $\tau$-formula $\varphi_R(x_1,...,x_{\mi{Ar}(R)})$ such that
we have
$$R^{F(\mf{B})}(a_1,...,a_{\mi{Ar}(R)})\ \Leftrightarrow\ \mf{B}\models\varphi_R\bigl(f(a_1),...,f(a_{\mi{Ar}(R)})\bigr)$$
for every tuple $(a_1,...,a_{\mi{Ar}(R)})\in(\mi{Dom}(F(\mf{B})))^{\mi{Ar}(R)}$. Here $\mi{Ar}(R)$ is the arity of $R$.
We then say that the class $\mathcal{A}$ is \emph{uniformly first-order
interpretable in $\mathcal{C}$}. If $\mathcal{A}$ is a singleton class $\{\mf{A}\}$, we say that $\mf{A}$ is \emph{uniformly first-order
interpretable in $\mathcal{C}$}.

Assume that a class of $\sigma$-structures $\mathcal{A}$ is uniformly first-order interpretable in a class $\mathcal{C}$ of $\tau$-structures. Let $\mathcal{P}$ be a set of
unary relation symbols such that $\mathcal{P}\cap (\sigma\cup\tau)\, =\, \emptyset$. Define a map $I$ from the set of first-order $(\sigma\cup\mathcal{P})$-formulae to
the set of first-order $(\tau\cup\mathcal{P})$-formulae as follows. 
\begin{enumerate}
\item
If $P\in\mathcal{P}$, then $I(Px)\, :=\, Px$.
\item
If $k\in\mathbb{N}_{\geq 1}$ and $R\in\sigma$ is a $k$-ary relation symbol,
then
$I(R(x_1,...,x_k))\, :=\, \varphi_{R}(x_1,...,x_k)$,
where $\varphi_{R}(x_1,...,x_k)$ is the first-order formula for $R$ witnessing the
fact that $\mathcal{A}$ is uniformly first-order 
interpretable in\, $\mathcal{C}$.
\item
$I(x=y)\, :=\, x=y$.
\item
$I(\neg\varphi):=\neg I(\varphi)$.
\item
$I(\varphi\wedge\psi)\, :=\, I(\varphi)\wedge I(\psi)$.
\item
$I\bigl(\exists x\, \psi(x)\bigr)\, :=\, \exists x\bigl(\varphi_{\mi{Dom}}(x)\wedge I(\psi(x))\bigr).$
\end{enumerate}
We call the map $I$ the \emph{$\mathcal{P}$-expansion of a uniform interpretation of\, $\mathcal{A}$ in $\mathcal{C}$}. When $\mathcal{A}$ and $\mathcal{C}$ are known
from the context, we may call $I$ simply a \emph{$\mathcal{P}$-interpretation}. In the case where $\mathcal{P}$ is empty, the map $I$ is a \emph{uniform
interpretation of\, $\mathcal{A}$ in $\mathcal{C}$}.
\begin{lemma}\label{uniforminterpretationlemma}
Let $\sigma$ and $\tau$ be finite relational vocabularies. Let $\mathcal{A}$ be a class of $\sigma$-structures and\, $\mathcal{C}$ a class of $\tau$-structures. 
Assume that $\mathcal{A}$ is uniformly first-order interpretable in $\mathcal{C}$. Let $\mathcal{P}$ be a set of unary relation symbols such
that $\mathcal{P}\cap(\sigma\cup\tau)=\emptyset$. Let $I$ denote a related $\mathcal{P}$-interpretation.
Let\, $\varphi$ be a first-order $(\sigma\cup\mathcal{P})$-sentence. The following conditions are equivalent. 
\begin{enumerate}
\item
There exists an expansion $\mf{A}^*$ of a structure $\mf{A}\in\mathcal{A}$ to the vocabulary $\sigma\cup\mathcal{P}$ such
that $\mf{A}^*\models \varphi$.

\item
There exists an expansion $\mf{B}^*$ of a structure $\mf{B}\in\mathcal{C}$ to the
vocabulary $\tau\cup\mathcal{P}$ such that
$\mf{B}^*\models I(\varphi)$. 
\comment
\item
There exists an expansion $\mf{A}^*$ of the structure $\mf{A}$ to the vocabulary $\sigma\cup\mathcal{P}$ such
that $\mf{A}^*\models \varphi$.
\item
For each $\mf{B}_1\in\mathcal{C}$, there exists some expansion $\mf{B}_1^*$ of\, $\mf{B}_1$ to the vocabulary $\tau\cup\mathcal{P}$ such
that $\mf{B}_1^*\models I(\varphi)$.
\item
There exists some $\mf{B}_2\in\mathcal{C}$ such that for some
expansion $\mf{B}_2^*$ of\, $\mf{B}_2$ to the
vocabulary $\tau\cup\mathcal{P}$, we
have $\mf{B}_2^*\models I(\varphi)$.
\endcomment
\end{enumerate}
\end{lemma}
\begin{proof}
Straightforward.
\end{proof}

\subsection{Logics and structures}\label{logics and structures}

\emph{Monadic second order logic}, $\mathrm{MSO}$,
extends first-order logic with
quantification of
relation symbols ranging over subsets of the domain of a model.
In \emph{universal (existential) monadic second order logic}, $\forall \mathrm{MSO}$ ($\exists \mathrm{MSO}$),
the quantification of monadic relations is restricted to universal (existential)
prenex quantification in the beginning of formulae.
The logics $\forall \mathrm{MSO}$ and $\exists \mathrm{MSO}$ are also known as monadic $\Pi_1^1$ and monadic $\Sigma_1^1$.
\emph{Weak monadic second-order logic}, $\mathrm{WMSO}$, is a semantic variant of
monadic second-order logic in which the quantified relation symbols range over
finite subsets of the domain of a model.
The weak variants $\forall \mathrm{WMSO}$
and $\exists \mathrm{WMSO}$ of $\forall \mathrm{MSO}$ and $\exists\mathrm{MSO}$ are
defined in the obvious way.
%

%

Let $\mathcal{L}$ be any fragment of second-order logic.
The \emph{$\mathcal{L}$-theory} of a structure $\mf{M}$ of a vocabulary $\tau$ is the set of $\tau$-sentences $\varphi$ of
$\mathcal{L}$ such that $\mf{M}\models\varphi$.

Define two binary relations $H,V\subseteq\mathbb{N}^2\times\mathbb{N}^2$ as follows.
\begin{itemize}
\item
$H\ =\ \{\ \bigl((i,j),(i+1,j)\bigr)\ |\ i,j\in\mathbb{N}\ \}$.
\item
$V\ =\ \{\ \bigl((i,j),(i,j+1)\bigr)\ |\ i,j\in\mathbb{N}\ \}$.
\end{itemize}
We let $\mf{G}$ denote the structure $(\mathbb{N}^2,H,V)$, and call
it the \emph{grid}. The relations $H$ and $V$ are
called the \emph{horizontal} and \emph{vertical successor relations} of $\mf{G}$,
respectively. A \emph{supergrid} is a structure of the vobabulary $\{H,V\}$ that has $\mf{G}$ as a
substructure. We denote the class of supergrids by $\mathcal{G}$.
Let $(\mf{G},R)$ be the expansion of $\mf{G}$, where
$R\ =\ \{\ \bigl((0,i),(0,j)\bigr)\in\mathbb{N}^2\times\mathbb{N}^2\ |\ i < j\ \}.$
We denote the structure $(\mf{G},R)$ by $\mf{R}$, and call it the \emph{recurrence grid}.

\begin{figure}
\centering
\includegraphics[scale=1]{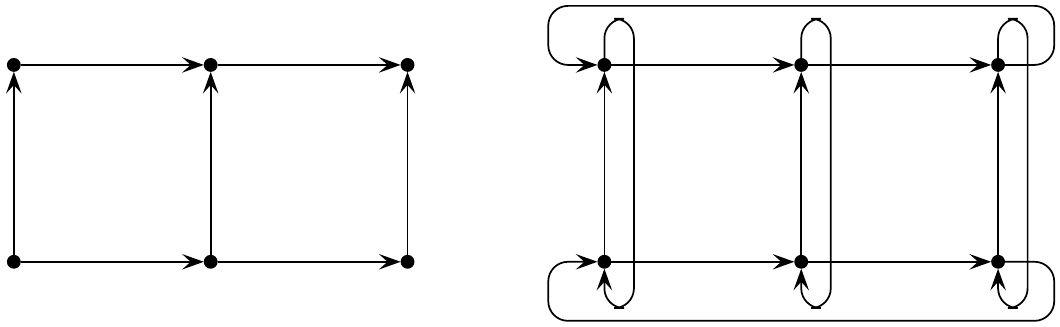}
\caption{Illustration of a $3\times 2$ grid and a $3\times 2$ torus.}
\label{fig:torus}
\end{figure}
Let $m$ and $n$ be positive integers. Define two binary relations $H_{m,n},V_{m,n}\subseteq (m\times n)^2$ as follows.
(Note that we define $m = \{0,...,m-1\}$, and analogously for $n$.)
\begin{itemize}
\item $H_{m,n}=H\upharpoonright (m\times n)^2\cup\{((m-1,i),(0,i))\mid i<n\}$.
\item $V_{m,n}=V\upharpoonright (m\times n)^2\cup\{((i,n-1),(i,0))\mid i<m\}$.
\end{itemize}
We call the structure $(m\times n, H_{m,n},V_{m,n})$ the \emph{$m\times n$ torus} and denote it by  $\mf{T}_{m,n}$.
A torus is essentially a finite grid whose east border wraps back to the
west border and north border back to the south border.

\subsection{Geometric affine betweenness structures}\label{structures}

Let $(\rsq, d)$ be the $n$-dimensional Euclidean space with the canonical metric $d$.
We always assume $n\geq 1$. We define the
ternary Euclidean \emph{betweenness} relation $\beta$ such that $\beta(s,t,u)$ iff $d(s,u)=d(s,t)+d(t,u)$. By $\beta^*$ we
denote the \emph{strict betweenness} relation, i.e., $\beta^*(s,t,u)$ iff $\beta(s,t,u)$ and $s\not=t\not=u$.
We say that the points $s,t,u\in\mathbb{R}^n$ are \emph{collinear} if the 
disjunction $\beta(s,t,u)\vee \beta(s,u,t)\vee \beta(t,s,u)$ holds in $\rp$.
We define the first-order $\{\beta\}$-formula
$collinear(x,y,z):=\beta(x,y,z)\vee\beta(x,z,y)\vee\beta(y,x,z).$

Below we study geometric betweenness structures of the type $(T,\beta_T)$ where $T\subseteq\mathbb{R}^n$ and $\beta_T=\beta\upharpoonright T$.
Here $\beta\upharpoonright T$ is the restriction of the betweenness predicate $\beta$ of $\mathbb{R}^n$ to the set $T$.
To simplify notation, we usually refer to these structures by $(T,\beta)$.
Let $T\subseteq\rsq$ and let $\beta$ the corresponding betweenness relation.
We say that $L\subseteq T$ is a \emph{line in $T$} if
the following conditions hold.
\begin{enumerate}
\item
There exist points $s,t\in L$ such that $s\not= t$.
\item
For all $s,t,u\in L$, the points $s,t,u$ are collinear.
\item
Let  $s,t\in L$ be points such that $s\not=t$. For all $u\in T$, if $\beta(s,u,t)$ or $\beta(s,t,u)$, then $u\in L$.
\end{enumerate}
Let $T\subseteq\rsq$ and let $L_1$ and $L_2$ be lines in $T$. We say that $L_1$ and $L_2$
\emph{intersect} if $L_1\not=L_2$ and $L_1\cap L_2\neq\emptyset$. We say that the lines $L_1$ and $L_2$
\emph{intersect in $\rsq$} if $L_1\not=L_2$ and $L_1'\cap L_2'\neq\emptyset$, 
where $L_1',L_2'$ are the lines in $\rsq$ such that $L_1\subseteq L_1'$ and $L_2\subseteq L_2'$.
A subset $S\subseteq\mathbb{R}^n$ is an \emph{$m$-dimensional flat} of $\mathbb{R}^n$, where $0\leq m\leq n$, if there exists a set of
$m$ linearly independent vectors $v_1,\dots,v_m\in\mathbb{R}^n$ and a vector $h\in \rsq$ such that $S$ is
the $h$-translated span of the vectors $v_1,\dots,v_m$,
in other words $S=\{u\in\rsq\mid u=h + r_1 v_1 + \dots +r_m v_m,\ r_1,\dots,r_m\in\mathbb{R}\}$.
None of the vectors $v_i$ is allowed to be the zero-vector. 
A set $U\subseteq\mathbb{R}^n$ is a \emph{linearly regular $m$-dimensional flat}, where $0\leq m\leq n$, if the following conditions hold.
\begin{enumerate}
\item
There exists an $m$-dimensional flat $S$ such that $U\subseteq S$. 
\item
There does not exist any $(m-1)$-dimensional flat $S$ such that $U\subseteq S$.
\item
$U$ is \emph{linearly complete}, i.e., if $L$ is a line in $U$ and $L'\supseteq L$ the corresponding line in $\mathbb{R}^n$,
and if $r\in L'$ is a point in $L'$ and $\epsilon\in\mathbb{R}_+$ a positive real number, then there exists a point $s\in L$
such that $d(s,r) < \epsilon$. Here $d$ is the canonical metric of $\mathbb{R}^n$.
\item
$U$ is \emph{linearly closed}, i.e., if $L_1$ and $L_2$ are lines in $U$ and $L_1$ and $L_2$ intersect in $\mathbb{R}^n$, 
then the lines $L_1$ and $L_2$ intersect. In other words, there exists a point $s\in U$ such that $s\in L_1\cap L_2$.
\end{enumerate} 
A set $T\subseteq\mathbb{R}^n$ \emph{extends linearly in $mD$}, where $m\leq n$, if
there exists a linearly regular $m$-dimensional flat $S$, a positive real number $\epsilon\in\mathbb{R}_+$ and a
point $x \in S\cap T$ such that 
$\{\ u\in S\ |\ d(x,u)<\epsilon\ \}\ \subseteq T.$
It is easy show that for example $\mathbb{Q}^2$ extends linearly in $2D$.

\subsection{Tilings}
A function $t:4\longrightarrow\mathbb{N}$ is called a \emph{tile type}.
Define the set $\mathrm{TILES}\ :=\ \{\ P_t\ |\ t\text{ is a tile type }\ \}$ of unary relation symbols.
The unary relation symbols in the set $\mathrm{TILES}$ are called \emph{tiles}.
The numbers $t(i)$ of a tile $P_t$ are the \emph{colours} of $P_t$.
The number $t(0)$ is the \emph{top colour}, $t(1)$ the \emph{right colour},
$t(2)$ the \emph{bottom colour}, and $t(3)$ the \emph{left colour} of $P_t$.
Let $T$ be a finite nonempty set of tiles. We say that a structure $\mf{A}=(A,V,H)$, where $V,H\subseteq A^2$, is
\emph{$T$-tilable}, if there exists an expansion of $\mf{A}$ to
the vocabulary
$\{H,V\}\cup\{\ P_t\ |\ P_t\in T\ \}$
such that the following conditions hold.
\begin{enumerate}
\item
Each point of $A$ belongs to the extension of
exactly one symbol $P_t$ in $T$.
\item
If $u H v$ for some points $u,v\in A$,
then the right colour of the tile $P_t$ s.t. $P_{t}(u)$ is the same as the left
colour of the tile $P_{t'}$ such that $P_{t'}(v)$.
\item
If $u V v$ for some points $u,v\in A$,
then the top colour of the tile $P_t$ s.t. $P_{t}(u)$ is the same as the bottom
colour of the tile $P_{t'}$ such that $P_{t'}(v)$.
\end{enumerate}
Let $t\in T$. We say that the grid $\mf{G}$ is \emph{$t$-recurrently $T$-tilable} if 
there exists an expansion of $\mf{G}$ to the vocabulary $\{H,V\}\cup\{\ P_t\ |\ t\in T\ \}$ such that
the above conditions $1 - 3$ hold, and additionally,
there exist infinitely many points $(0,i)\in\mathbb{N}^2$
such that $P_t\bigl((0,i)\bigr)$. 
Intuitively this means that the tile $P_t$ occurs infinitely many 
times in the leftmost column of the grid $\mf{G}$.
Let $\mathcal{F}$ be the set of finite, nonempty sets $T\subseteq\mathrm{TILES}$, and let $\mathcal{H}\ :=\ \{\ (t,T)\ |\ T\in\mathcal{F},\ t\in T\ \}$.
Define the following languages
\begin{align*}
\mathcal{T}\ :=&\ \{\ T\in\mathcal{F}\ \mid\ \mf{G}\text{ is $T$-tilable } \},\\
\mathcal{R}\ :=&\ \{\ (t,T)\in\mathcal{H}\ \mid\ \mf{G}\text{ is $t$-recurrently $T$-tilable } \},\\
\mathcal{S}\ :=&\ \{\ T\in \mathcal{F}\ \mid\
\text{ there is a torus $\mf{D}$ which is $T$-tilable }\}.
\end{align*}
The \emph{tiling problem} is the membership problem of the set $\mathcal{T}$
with the input set $\mathcal{F}$.
The \emph{recurrent tiling problem} is the
membership problem of the set $\mathcal{R}$
with the input set $\mathcal{H}$.
The
\emph{periodic tiling problem} is the
membership problem of $\mathcal{S}$ with the input set $\mathcal{F}$.

\begin{theorem}\cite{Berger}
The tiling problem is $\Pi_1^0$-complete.
\end{theorem}
\begin{theorem}\cite{Harel:1985}
The recurrent tiling problem is $\Sigma_1^1$-complete.
\end{theorem}

\begin{theorem}\cite{GurevichKoryakov:1972}\label{periodictilingcomplete}
The periodic tiling problem is $\Sigma^0_1$-complete.
\end{theorem}

\begin{lemma}\label{tilingdefinablelemma}
There is a computable function associating
each input $T$ to the
(periodic) tiling problem with a first-order sentence $\varphi_{T}$ of the
vocabulary $\tau:=\{H,V\}\cup T$ 
such that for all structures $\mf{A}$ of the vocabulary $\{H,V\}$, the structure $\mf{A}$ is $T$-tilable iff
there exists an expansion $\mf{A}^*$ of\, $\mf{A}$ to the
vocabulary $\tau$ such that $\mf{A}^*\models\varphi_{T}$.
\end{lemma} 
\begin{proof}
Straightforward.
\end{proof}
\begin{lemma}\label{recurrencetilingdefinablelemma}
There is a computable function associating
each input $(t,T)$ of the
recurrent tiling problem with a first-order sentence $\varphi_{(t,T)}$ of the
vocabulary $\tau:=\{H,V,R\}\cup T$ 
such that the grid $\mf{G}$ is $t$-recurrently $T$-tilable iff
there exists an expansion $\mf{R}^*$ of the recurrence grid $\mf{R}$ to the
vocabulary $\tau$ such that $\mf{R}^*\models\varphi_{(t,T)}$.
\end{lemma} 
\begin{proof}
Straightforward.
\end{proof}
It is easy to see that the grid $\mf{G}$ is $T$-tilable iff there exists a supergrid $\mf{G'}$ that is $T$-tilable.
\section{Expressivity of universal $\mathrm{MSO}$ and
weak universal $\mathrm{MSO}$ over affine real structures $(\mathbb{R}^n,\beta)$}\label{section3}
In this section we investigate the expressive powers of $\forall{\mathrm{WMSO}}$
and $\forall\mathrm{MSO}$. While it is rather easy to conclude that the two logics
are incomparable in a rather strong sense (see Proposition \ref{expressivityproposition}),
when attention is limited to structures $(\rpm,(R_i)_{i\in I})$ that expand the affine real structure $\rp$,
sentences of $\forall\mathrm{WMSO}$ translate into equivalent sentences of $\forall\mathrm{MSO}$.
Let $\mathcal{L}$ and $\mathcal{L}'$ be fragments of second-order logic.
We write $\mathcal{L}\leq\mathcal{L}'$, if for every vocabulary $\sigma$,
any class of $\sigma$-structures definable by a $\sigma$-sentence of $\mathcal{L}$ is
also definable by a $\sigma$-sentence of $\mathcal{L}'$.
Let $\tau$ be a vocabulary such that $\beta\not\in\tau$.
The class of all expansions of $\rp$ to the vocabulary $\{\beta\}\cup\tau$ is
called the class of \emph{affine real $\tau$-structures}. Such structures can be
regarded as $\tau$-structures \emph{embedded} in the geometric structure $\rp$.
We say that \emph{$\mathcal{L}\leq\mathcal{L'}$ over $\rp$}, if 
for every vocabulary $\tau$ s.t. $\beta\not\in\tau$,
any subclass definable w.r.t. the class $\mathcal{C}$ of all
affine real $\tau$-structures by a sentence of $\mathcal{L}$ is also
definable w.r.t. $\mathcal{C}$ by a sentence of $\mathcal{L}'$.
We sketch a canonical proof of the following very simple observation. The result $\forall\mathrm{WMSO}\not\leq\mathrm{MSO}$
follows from already existing results (see \cite{tenCate:2011} for example), and the result $\forall\mathrm{MSO}\not\leq\mathrm{WMSO}$ is easy to prove.
\begin{proposition}\label{expressivityproposition}
$\forall\mathrm{WMSO}\not\leq\mathrm{MSO}$ and $\forall\mathrm{MSO}\not\leq\mathrm{WMSO}$.
\end{proposition}
\begin{proof}[Proof Sketch]
It is easy to observe that $\forall\mathrm{WMSO}\not\leq\mathrm{MSO}$: consider the
sentence $\forall X\exists y\, \neg Xy$. This $\forall\mathrm{WMSO}$ sentence is
true in a model iff the domain of the model is infinite. A straightforward
monadic second-order Ehrenfeucht-Fra\"{i}ss\'{e} game
argument can be used to establish that
infinity is not expressible by any $\mathrm{MSO}$ sentence.
To show that $\forall\mathrm{MSO}\not\leq\mathrm{WMSO}$, consider the
structures $(\mathbb{R},<)$ and $(\mathbb{Q},<)$. The structures can be
separated by a sentence of $\forall\mathrm{MSO}$
stating that every subset bounded
from above has a least
upper bound. To see that the two structures cannot be
separated by any sentence of $\mathrm{WMSO}$, consider the
variant of the $\mathrm{MSO}$ Ehrenfeucht-Fra\"{i}ss\'{e} game
where the players choose \emph{finite sets} in
addition to domain elements. It is easy to establish
that this game characterizes the expressivity of $\mathrm{WMSO}$.
To see that the duplicator has a winning strategy in a game of any finite
length played on the
structures $(\mathbb{R},<)$ and $(\mathbb{Q},<)$, we
devise an extension of the folklore winning strategy in the 
corresponding first-order game. Firstly, the duplicator can obviously 
always pick an element whose  
betweenness configuration corresponds exactly to that of the element picked by the spoiler.
Furthermore, even if the spoiler picks a finite set,
it is easy to see that the duplicator can pick his set such that
each of its elements respect the betweenness
configuration of the set
picked by the spoiler.  
\end{proof}
We then show that $\forall\mathrm{WMSO}\leq\forall\mathrm{MSO}$ and $\mathrm{WMSO}\leq\mathrm{MSO}$ over $(\mathbb{R}^n,\beta)$
for any $n\geq 1$.
\begin{theorem}[Heine-Borel]
A set $S\subseteq\rsq$ is closed and bounded iff every open cover of $S$ has a
finite subcover. 
\end{theorem}
\begin{theorem}
\label{expressivitytheorem}
Let $\mathcal{C}$ be the class of expansions $(\rpm,P)$ of $\rp$ with a unary predicate $P$,
and let $\mathcal{F}\subseteq\mathcal{C}$ be the subclass of $\mathcal{C}$ where $P$ is finite.
The class $\mathcal{F}$ is first-order definable with respect to $\mathcal{C}$. 
\end{theorem}
\begin{proof}
We shall first establish that a set $T\subseteq\rsq$ is finite
iff it is closed, bounded and
consists of isolated points of $T$. Recall that an
isolated point $u$ of a set $U\subseteq\rsq$ is a point
such that there exists some open ball $B$ such that $B\cap U =\{u\}$.

Assume $T\subseteq\rsq$ is finite. Since $T$ is finite, we can find a minimum 
distance between points in the set $T$.
Therefore it is clear that each point $t$ in $T$ belongs to some open ball $B$
such that $B\cap T = \{t\}$, and hence $T$
consists of isolated points.
Similarly, since $T$ is finite, each point $b$ in the complement of $T$
has some minimum distance to the points of $T$, and therefore $b$
belongs to some open ball $B\subseteq\rsq\setminus T$.
Hence the set $T$ is the complement of the union of
open balls $B$ such that $B\subseteq\rsq\setminus T$, and therefore $T$ is closed.
Finally, since $T$ is finite, we can find a
maximum distance between the points in $T$,
and therefore $T$ is bounded.
Assume then that $T\subseteq\rsq$ is closed, bounded and consists of isolated points of $T$.
Since $T$ consists of isolated points, it has an open cover $\mathcal{C}\subseteq\mathrm{Pow}(\rsq)$ such
that each set in $\mathcal{C}$ contains exactly one point $t\in T$. The set $\mathcal{C}$ is an open
cover of $T$, and by the Heine-Borel theorem, there exists a
finite subcover $\mathcal{D}\subseteq\mathcal{C}$ of the set $T$.
Since $\mathcal{D}$ is finite and each set in $\mathcal{D}$ contains exactly one point of $T$, 
the set $T$ must also be finite.
We then conclude the proof by establishing that there exists a first-order formula $\varphi(P)$
stating that the unary predicate $P$ is closed, bounded and consists of isolated points. We will first define a formula $\mi{parallel}(x,y,t,k)$ stating that the
lines defined by $x,y$ and $t,k$ are parallel in $\rp$. We define
\begin{align*}
&\mi{parallel}(x,y,t,k):= x\neq y\wedge t\neq k \wedge \Big((\mi{collinear}(x,y,t)\wedge \mi{collinear}(x,y,k))\\
&\quad\vee \big(\neg\exists z(\mi{collinear}(x,y,z)\wedge\mi{collinear}(t,k,z))\\
&\quad\quad\wedge \exists z_1 z_2(x\neq z_1\wedge\mi{collinear}(x,y,z_1)\wedge\mi{collinear}(x,t,z_2)\wedge\mi{collinear}(z_1,z_2,k))\, \big)\Big).
\end{align*}
We will then define first-order $\{\beta\}$-formulae $\mi{basis}_k(x_0,\dots,x_k)$ and $\mi{flat}_k(x_0,\dots,x_k,z)$ using simultaneous recursion.
The first formula states that the vectors corresponding to the pairs $(x_0,x_i)$, $1\leq i\leq k$, form a basis of a $k$-dimensional flat.
The second formula states the points $z$ are exactly the points in the span of the basis defined by the vectors $(x_0,x_i)$, the origin being $x_0$.
First define
$\mi{basis}_0(x_0):= x_0=x_0$ and
$\mi{flat}_0(x_0,z):= x_0=z$.
Then define $\mi{flat}_k$ and $\mi{basis_k}$ recursively in the following way.
\begin{align*}
&\mi{basis}_k(x_0,\dots,x_k):=\mi{basis}_{k-1}(x_0,\dots,x_{k-1})\wedge \neg \mi{flat}_{k-1}(x_0,\dots,x_{k-1},x_k),\\
&\mi{flat}_k(x_0,\dots,x_k,z):=\mi{basis}_k(x_0,\dots,x_k)\\
&\quad\wedge \exists y_0,\dots,y_k\Big(y_0=x_0 \wedge y_k=z
\wedge\bigwedge_{i\, \leq\, k-1} \big(y_i=y_{i+1}\vee\mi{parallel}(x_0,x_{i+1},y_i,y_{i+1})\big)\Big).
\end{align*}
We then define a first-order $\{\beta,P\}$-formula $\mathit{sepr}(x,P)$ 
asserting that $x$ belongs to an open ball $B$ such that each point in $B\setminus\{x\}$
belongs to the complement of $P$.
The idea is to state that there exist $n+1$ points $x_0,\dots,x_n$
that form an \emph{$n$-dimensional triangle} around $x$, and every point contained in the triangle (with $x$ being a
possible exception) belongs to the complement of $P$. Every open ball in $\rsq$ is contained in some $n$-dimensional triangle in $\rsq$ and vice versa. We
will recursively define first-order formulae $\mi{opentriangle}_k(x_0,\dots,x_k,z)$ stating that $z$ is properly inside a $k$-dimensional triangle
defined by $x_0,\dots,x_k$. First define $\mi{opentriangle}_1(x_0,x_1,z):=\beta^*(x_0,z,x_1)$, and then define
\begin{align*}
&\mi{opentriangle}_{k}(x_0,\dots,x_k,z):=\mi{basis}_k(x_0,\dots,x_k)\\
&\quad\wedge \exists y\big(\mi{opentriangle}_{k-1}(x_0,\dots,x_{k-1},y)\wedge \beta^*(y,z,x_k)\big).
\end{align*}
\begin{figure}
\centering
\includegraphics[scale=2]{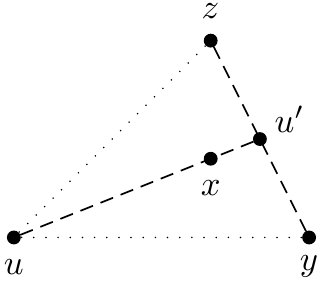}
\caption{Illustration of $\mi{opentriangle}_2(y,z,u,x)$.}
\label{fig:opentriangle}
\end{figure}
We are now ready to define $\mathit{sepr}(x,P)$. Let
\begin{multline*}
\mathit{sepr}(x,P) := \exists x_0,\dots,x_n\Bigl(\mi{opentriangle}_n(x_0,\dots, x_n,x)\\
\wedge \forall y \big((\mi{opentriangle}_n(x_0,\dots,x_n,y)\wedge y\neq x)\rightarrow \neg Py\big)\Bigr).
\end{multline*}
Now, the sentence
$\varphi_1:=\forall x\bigl(\neg Px\ \rightarrow\ \mathit{sepr}(x,P)\bigr)$
states that each point in the complement of $P$ is contained in an open ball $B\subseteq\rsq\setminus P$.
The sentence therefore states that the complement of $P$ is a union of open balls.
Since the set of unions of open balls is exactly the same as the set of open sets, the sentence states that $P$ is closed.
The sentence
$\varphi_2:=\forall x\bigl(Px\ \rightarrow\ \mathit{sepr}(x,P)\bigr)$
clearly states that $P$ consists of isolated points.
Finally, in order to state that $P$ is bounded, we define a formula asserting that
there exist points $x_0,\dots,x_n$ that
form an n-dimensional triangle around $P$.
\begin{multline*}
\varphi_3:=\exists x_0,\dots,x_n\Bigl(\mi{basis}_n(x_0,\dots,x_n)
\wedge \forall y\bigl(Py\rightarrow \mi{opentriangle}_n(x_0,\dots,x_n,y)\bigr)\Bigr)
\end{multline*} 
The conjunction $\varphi_1\wedge\varphi_2\wedge\varphi_3$
states that $P$ is finite.
\end{proof}
\begin{corollary}
Limit attention to expansions of $\rp$.
Sentences of\, $\forall\mathrm{WMSO}$ translate into
equivalent sentences of $\forall\mathrm{MSO}$,
and sentences of $\mathrm{WMSO}$ into equivalent sentences of\, $\mathrm{MSO}$. 
\end{corollary}


\section{Undecidable theories of geometric structures with an affine betweenness relation}\label{main}

In this section we prove that the universal monadic second-order theory of any geometric structure $(T,\beta)$ that extends linearly in $2D$ is undecidable. In
addition we show that the universal monadic second-order theories of structures $(\rsq,\beta)$ with $n\geq 2$ are highly undecidable.
In fact, we show that the theories of structures extending linearly in $2D$ are $\Sigma_1^0$-hard,
while the theories of structures $(\rsq,\beta)$ with $n\geq 2$ are $\Pi_1^1$-hard---and therefore not even arithmetical. 
We establish the results by a reduction from the (recurrent) tiling problem to the problem of deciding
whether a particular $\{\beta\}$-sentence of monadic $\Sigma_1^1$ is \emph{satisfied} by $(T,\beta)$ (respectively, $\rp$).
The argument is based on interpreting supergrids in corresponding $\{\beta\}$-structures.
%
%
%

\subsection{Lines and sequences}
Let $T\subseteq \rsq$. Let $L$ be a line in $T$.
Any nonempty subset $Q$ of $L$ is called a
\emph{sequence} in $T$.
Let $E\subseteq T$ and $s,t\in T$. If $s\not=t$ and if $u\in E$ for all points $u\in T$ such that $\beta^*(s,u,t)$, we say that the 
points $s$ and $t$ are \emph{linearly $E$-connected} (in $(T,\beta)$). 
If there exists a point $v\in T\setminus E$ such that $\beta^*(s,v,t)$,
we say that $s$ and $t$ are \emph{linearly disconnected with respect to $E$} (in $(T,\beta)$).
\begin{definition}\label{discretelyspaceddfn}
Let $Q$ be a sequence in $T\subseteq \rsq$. Suppose that for each $s,t\in Q$ such that $s\not=t$, there exists a point $u\in T\setminus\{s\}$ such that 
\begin{enumerate}
\item
$\beta(s,u,t)$ and
\item
$\forall r\, \in\, T\ \bigl(\, \beta^*(s,r,u)\rightarrow r\not\in Q\, \bigr)$, i.e., the points $s$ and $u$ are linearly $(T\setminus Q)$-connected.
\end{enumerate}
Then we call $Q$ a $\emph{discretely spaced sequence in T}$.
\end{definition}
\begin{definition}\label{discretelyinfinitedfn}
Let $Q$ be a discretely spaced sequence in $T\subseteq \rsq$. Assume that there exists a point $s\in Q$ such that for each point $u\in Q$, there exists a point
 $v\in Q\setminus\{u\}$ such that $\beta(s,u,v)$.
Then we call the sequence $Q$ a \emph{discretely infinite sequence in $T$}. The point $s$ is
called a \emph{base point} of $Q$. 
\end{definition}
\begin{definition}\label{sequencewithzero}
Let $Q$ be a sequence in $T\subseteq \rsq$. Let $s\in Q$ be a
point such that there do not exist
points $u,v\in Q\setminus\{s\}$ such that $\beta(u,s,v)$.
Then we call $Q$ a \emph{sequence in $T$ with a zero}.
The point $s$ is a \emph{zero-point of $Q$}.
Notice that $Q$ may have up to two
zero-points.
\end{definition}
It is easy to see that a discretely infinite sequence has at most one zero point.
\begin{definition}\label{omegalikedfn}
Let $Q$ be a discretely infinite sequence in $T\subseteq\rsq$ with a zero.
Assume that for each $r\in T$ such that there exist points $s,u\in Q\setminus\{r\}$
with $\beta(s,r,u)$, there also exist points $s',u'\in Q\setminus\{r\}$ such that
\begin{enumerate}
\item
$\beta(s',r,u')$ and
\item
$\forall v\, \in\, T\setminus\{r\}\ \bigl(\, \beta^*(s',v,u')\rightarrow v\not\in Q\, \bigr)$.
\end{enumerate}
Then we call $Q$ an \emph{$\omega$-like sequence in $T$} (cf. Lemma \ref{nonstandardmodelsexcluded}).
\end{definition}
%
%
%
%
%
%
\begin{lemma}\label{omegalikesequencedefinability}
Let $P$ be a unary relation symbol. There is a first-order sentence $\varphi_{\omega}(P)$ of the vocabulary $\{\beta,P\}$ such that
for every $T\subseteq \rsq$ and for every expansion $(T,\beta,P)$ of $(T,\beta)$, we have $(T,\beta,P)\models\varphi_{\omega}(P)$ if and only if the
interpretation of $P$ is an $\omega$-like sequence in $T$.  
\end{lemma}
\begin{proof} 
Define
\[
\mi{sequence}(P):=\exists x\, Px\,
\wedge\, \forall x\forall y\forall z\,
\bigl(Px\wedge Py \wedge Pz\ \rightarrow\ \mi{collinear}(x,y,z)\bigr).
\]
The formula $sequence(P)$ states that $P$ is a sequence. By inspection of Definition \ref{discretelyspaceddfn}, it is
easy to see that there is a first-order formula $\psi$ such that the conjunction $sequence(P)\wedge\psi$ states
that $P$ is a discretely spaced sequence. Continuing this trend, it is straightforward to observe
that Definitions \ref{discretelyinfinitedfn}, \ref{sequencewithzero} and \ref{omegalikedfn} specify
first-order properties, and therefore there exists a first-order formula $\varphi_{\omega}(P)$ stating
that $P$ is an $\omega$-like sequence.
\end{proof}
\begin{definition}\label{successordefinition}
Let $P$ be a sequence in $T\subseteq \rsq$ and $s,t\in P$. The points $s,t$ are
called \emph{adjacent} with respect to $P$, if the
points are linearly $(T\setminus P)$-connected.
Let $E\subseteq P\times P$ be the set of pairs $(u,v)$ such that 
\begin{enumerate}
\item
$u$ and $v$ are adjacent with respect to $P$, and 
\item
$\beta(z,u,v)$ for some zero point $z$ of $P$.
\end{enumerate}
We call $E$ the \emph{successor relation} of $P$.
\end{definition}
We let $\mi{succ}$ denote the successor relation of $\mathbb{N}$, i.e.,
$\mi{succ}:=\{\ (i,j)\in\mathbb{N}\times\mathbb{N}\ |\ i+1 = j\ \}.$

\begin{lemma}\label{nonstandardmodelsexcluded}
Let $P$ be an $\omega$-like sequence in $T\subseteq\rsq$ and $E$ the successor relation of $P$.
There is an embedding from $(\mathbb{N},\mi{succ})$ into $(P,E)$ such that $0\in\mathbb{N}$ maps to the zero point of $P$.
If $T=\rsq$, then $(\mathbb{N},\mi{succ})$ is isomorphic to $(P,E)$.
\end{lemma}
\begin{proof}
We denote by $i_0$ the unique zero point of $P$.
Since $P$ is a discretely infinite sequence, it has a base point.
Clearly $i_0$ has to be the only base point of $P$.
It is straightforward to establish that since $P$ is an $\omega$-like sequence with the base point $i_0$, there exists a sequence
 $(a_i)_{i\in \mathbb{N}}$ of points $a_i\in P$ such that $i_0=a_0$ and $a_{i+1}$ is the unique $E$-successor of $a_{i}$ for all $i\in \mathbb{N}$.
Define the function $h:\mathbb{N}\rightarrow P$ such that $h(i) = a_i$ for all $i\in\mathbb{N}$.
It is easy to see that $h$ is an embedding of $(\mathbb{N},\mi{succ})$ into $(P,E)$.
Assume then that $T=\rsq$. We shall show that the function $h:\mathbb{N}\longrightarrow P$ is a surjection.
Let $d$ denote the canonical metric of $\mathbb{R}$, and let $d_R$ be the
restriction of the canonical metric of $\rsq$ to the line $R$ in $\rsq$ such that $P\subseteq R$.
Let $g:\mathbb{R}\longrightarrow R$ be the isometry from $(\mathbb{R},d)$ to $(R,d_R)$ such
that $g(0) = i_0 = h(0)$ and such
that for all $r\in\mi{ran}(h)$, we have $\beta\bigl(i_0,g(1),r\bigr)$ or $\beta\bigl(i_0,r,g(1)\bigr)$.
Let $(R,\leq^R)$ be the
structure, where
$\leq^{R}\ =\ \{\ (u,v)\in R\times R\ |\ g^{-1}(u)\, \leq^{\mathbb{R}}\, g^{-1}(v)\ \}$.
If $\mi{ran}(h)$ is not bounded from above w.r.t. $\leq^R$, then $h$
must be a surjection. Therefore
assume that $\mi{ran}(h)$ is bounded above. By the Dedekind completeness of the reals, there
exists a least upper bound $s\in R$ of $\mi{ran}(h)$ w.r.t. $\leq^R$.
Notice that since $h$ is an 
embedding of $(\mathbb{N},\mi{succ})$ into $(P,E)$, we have $s\not\in\mi{ran}(h)$.
Due to the definition of $E$, it is sufficient to show that $\{\, t\in P\ |\ s\leq^R t\ \} = \emptyset$ in order to conclude
that $h$ maps onto $P$.
Assume that the least upper bound $s$ belongs to the set $P$. Since $P$ is a discretely spaced sequence,
there is a point $u\in\mathbb{R}^n\setminus\{s\}$ such that $\beta(s,u,i_0)$ and
$\forall r\in\mathbb{R}^n
\bigl(\beta^*(s,r,u)
\rightarrow r\not\in P\bigl)$.
Now $u<^R s$ and the points $u$ and $s$ are linearly $(\mathbb{R}^n\setminus P)$-connected,
implying that $s$ cannot be the
least upper bound of $\mi{ran}(h)$. This is a
contradiction. Therefore $s\not\in P$.
Assume, ad absurdum, that there exists a point $t\in P$
such that $\beta(i_0,s,t)$.
Now, since $P$ is an $\omega$-like sequence, there
exists points $u',v'\in P\setminus\{s\}$
such that
$\beta(u',s,v')$ and $\forall r\in \rsq\bigl(\beta^*(u',r,v')\rightarrow r\not\in P\bigr)$.
We have $\beta(s,u',i_0)$ or $\beta(s,v',i_0)$. Assume, by symmetry, that $\beta(s,u',i_0)$.
Now $u'<^R s$, and the points $u'$ and $s$ are
linearly $(\rsq\setminus P)$-connected.
Hence, since  $s\not\in \mi{ran}(h)$, we conclude that
$s$ is not the least upper bound of $\mi{ran}(h)$.
This is a contradiction.
\end{proof}
%
%
%
\subsection{Geometric structures $(T,\beta)$ with an undecidable monadic $\Pi_1^1$-theory}
\begin{figure}
\centering
\includegraphics[scale=1]{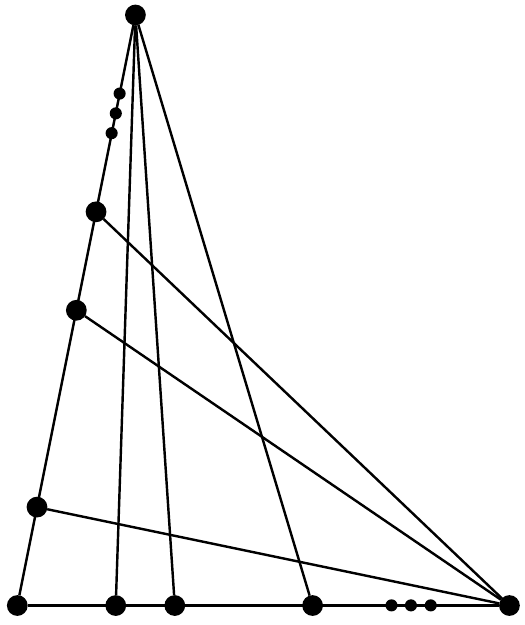}
\caption{Illustration of how the grid is interpreted in a Cartesian frame.}
\label{fig:triangle}
\end{figure}
Let $Q$ be an $\omega$-like sequence in $T\subseteq \rsq$ and let $q_0$ be the unique zero point of $Q$. Assume there exists a point $q_e\in T\setminus Q$
such that $\beta(q_0,q,q_e)$ holds for all $q\in Q$. We call $Q\cup\{q_e\}$ \emph{an $\omega$-like sequence with an endpoint in $T$}. The point $q_e$ is
the \emph{endpoint of $Q\cup\{q_e\}$}. Notice that the endpoint $q_e$ is the only point $x$ in $Q\cup\{q_e\}$ such that the following conditions hold.
\begin{enumerate}
\item
There does not exist points $s,t\in Q\cup\{q_e\}$ such that $\beta^*(s,x,t)$.
\item
$\forall y z \in Q\cup\{q_e\}\Bigl(\, \beta^*(x,y,z)\rightarrow \exists v\in Q\cup\{q_e\}\bigl(\beta^*(x,v,y)\Bigr)$.
\end{enumerate}
\begin{definition}
Let $P$ and $Q$ be $\omega$-like sequences with an endpoint in $T\subseteq \rsq$.
Let $p_e$ and $q_e$ be the endpoints of $P$ and $Q$, respectively. Assume that the following conditions hold.
\begin{enumerate}
\item There exists a point $z\in P\cap Q$ such that $z$ is the zero-point of both $P\setminus\{p_e\}$ and $Q\setminus\{q_e\}$.
\item There exists lines $L_P$ and $L_Q$ in $T$ such that $L_P\neq L_Q$, $P\subseteq L_P$ and $Q\subseteq L_Q$.
\item For each point $p\in P\setminus\{p_e\}$ and $q\in Q\setminus\{q_e\}$, the unique lines $L_p$ and $L_q$ in $T$ such
that $p,q_e\in L_p$ and $q,p_e\in L_q$ intersect.
\end{enumerate} 
We call the structure $(T,\beta,P,Q)$ a
\emph{Cartesian frame}.
\end{definition}
\begin{lemma}\label{boxedcartesianpicturesdefinablelemma}
Let $T\subseteq\rsq$, $n\geq 2$, and let $\mathcal{C}$ be the class of all expansions $(T,\beta,P,Q)$ of $(T,\beta)$ by
unary relations $P$ and $Q$.
The class of Cartesian frames with the domain $T$ is definable with
respect to $\mathcal{C}$ by a first-order sentence.
\end{lemma}
\begin{proof}
Straightforward by virtue of Lemma \ref{omegalikesequencedefinability}.
\end{proof}
\begin{lemma}\label{gridinterpretablelemma}
Let $T\subseteq\mathbb{R}^n$, $n\geq 2$. Let $\mathcal{C}$ be the class of
Cartesian frames with the domain $T$, and assume that $\mathcal{C}$ is nonempty.
Let $\mathcal{G}$ be the class of supergrids and $\mf{G}$ the grid.
There is a class $\mathcal{A}\subseteq \mathcal{G}$ that is uniformly
first-order interpretable in the class $\mathcal{C}$,
and furthermore, $\mf{G}\in\mathcal{A}$.
\end{lemma}
\begin{proof}
Let $\mf{C} = (T,\beta,P,Q)$ be a Cartesian frame. Let $p_e\in P$ and $q_e\in Q$ be the endpoints of $P$ an $Q$, respectively.
We shall interpret a supergrid $\mf{G}_{\mf{C}}$ in the Cartesian frame $\mf{C}$.
The domain of the interpretation of $\mf{G}_{\mf{C}}$ in $\mf{C}$ will be the set of points where the lines that
connect the points of $P\setminus\{p_e\}$ to $q_e$ and the
lines that connect the
points of $Q\setminus\{q_e\}$ to $p_e$ intersect.
First let us define the following formula which states in $\mf{C}$ that $x$ is the endpoint of $P$.
\[
\mi{end}_P(P,Q,x)\ :=\ Px \wedge \neg Qx \wedge \neg \exists y \exists z\bigr(Py\wedge Pz\wedge \beta^*(y,x,z)\bigl)
\]
In the following, we let atomic expressions of the type $x\not= p_e$ and $\beta^*(x,y,q_e)$ abbreviate
corresponding first-order formulae $\exists z\bigl(\mi{end}_P(P,Q,z)\wedge x\not= z\bigr)$ and $\exists z\bigl(\mi{end}_Q(Q,P,z)\wedge \beta^*(x,y,z)\bigr)$ of
the vocabulary $\{\beta,P,Q\}$ of $\mf{C}$. We define 
\begin{align*}
\varphi_{\mi{Dom}}(u)\ :=&\ u\not= p_e\wedge u\not= q_e\\
& \wedge \Big(Pu\vee Qu \vee\ \exists x y\big(Px\wedge x\not= p_e
\wedge Qy\wedge y\not= q_e\wedge \beta(x,u,q_e)\wedge\beta(y,u,p_e)\big)\Big),\\
\varphi_{H}(u,v)\ :=&\ 
\exists x \bigl(Qx\, \wedge\, \beta(x,u,v)\, \wedge\, \beta^*(u,v,p_e)\bigr)
\wedge\, \forall r\bigl(\, \beta^*(u,r,v)\, \rightarrow\, \neg \varphi_{\mi{Dom}}(r)\, \bigr),\\
\varphi_{V}(u,v)\ :=&\ 
\exists x \bigl(Px\, \wedge\, \beta(x,u,v)\, \wedge\, \beta^*(u,v,q_e)\bigr)
\wedge\, \forall r\bigl(\, \beta^*(u,r,v)\, \rightarrow\, \neg \varphi_{\mi{Dom}}(r)\, \bigr).\\
\end{align*}
Call $D_\mf{C} := \{\ r\in T\ |\ \mf{C}\models\varphi_{\mi{Dom}}(r)\ \}$ and
define the
structure $\mf{D}_\mf{C} = (D_\mf{C},H^{\mf{D}_\mf{C}},V^{\mf{D}_\mf{C}})$, where 
$$H^{\mf{D}_\mf{C}}\ :=\ \{\ (s,t)\in D_\mf{C}\times D_\mf{C}\ |\ \mathfrak{C}\models\varphi_H(s,t)\ \},$$
and analogously for $V^{\mf{D}_\mf{C}}$.
By Lemma \ref{nonstandardmodelsexcluded}, it is easy to see that there
exists an injection $f$ from the domain of the grid
$\mf{G}=(G,H,V)$ to $D_\mf{C}$ such that 
the following three conditions hold for all $u,v\in G$.
\begin{enumerate}
\item
$(u,v)\in H\ \Leftrightarrow\ \varphi_{H}\bigl(f(u),f(v)\bigr)$,
\item
$(u,v)\in V\ 
\Leftrightarrow\ \varphi_{V}\bigl(f(u),f(v)\bigr)$.
\end{enumerate}
Hence there is a supergrid $\mf{G}_\mf{C}=(G_\mf{C},H,V)$ such that
there exists an isomorphism $f$
from $G_\mf{C}$ to $D_\mf{G}$ such that the
above two conditions hold.

Let $\mathcal{A}:=\{\ \mf{G}_{\mf{C}}\in \mathcal{G}
\ |\ \mf{C}\text{ is a Cartesian frame with the domain $T$\ } \}$.
Clearly $\mf{G}\in\mathcal{A}$, and 
furthermore, $\mathcal{A}$ is uniformly
first-order interpretable in the
class of Cartesian frames with the domain $T$.
\end{proof}
\begin{lemma}\label{recurrencegridinterpr}
Let $n\geq 2$ be an integer. The recurrence grid $\mf{R}$ is
uniformly first-order interpretable in the class of
Cartesian frames with the domain $\mathbb{R}^n$.
\end{lemma}
\begin{proof}
Straightforward by Lemma \ref{nonstandardmodelsexcluded} and the proof of Lemma \ref{gridinterpretablelemma}.
\end{proof}
\begin{theorem}\label{infinitelemma}
Let $T\subseteq \rsq$ be a set and let $\beta$ be the corresponding betweenness relation. Assume that $T$ extends linearly in $2D$.
The monadic $\Pi_1^1$-theory of $(T,\beta)$ is $\Sigma_1^0$-hard.
\end{theorem}
\begin{proof}
Since $T$ extends linearly in $2D$, we have $n\geq 2$. 
Let $\sigma = \{H,V\}$ be the vocabulary of
supergrids, and let $\tau = \{\beta,X,Y\}$ be the vocabulary of
Cartesian frames.
By Lemma \ref{boxedcartesianpicturesdefinablelemma}, 
there exists a first-order $\tau$-sentence that defines the class of Cartesian frames  with
the domain $T$ with respect to the class of all
expansions of $(T,\beta)$ to
the vocabulary $\tau$. Let $\varphi_{\mi{Cf}}$ denote
such a sentence.
By Lemma \ref{tilingdefinablelemma}, there is a
computable function that associates each input $S$ to the tiling problem with a
first-order $\sigma\cup S$-sentence $\varphi_{S}$
such that a structure $\mf{A}$ of the vocabulary $\sigma$ is $S$-tilable if and only if there is an
expansion $\mf{A}^*$ of the structure
$\mf{A}$ to the vocabulary $\sigma\cup S$ such that 
$\mf{A}^*\models\varphi_{S}$.
Since $T$ extends linearly in $2D$, the class of Cartesian frames with the domain $T$ is nonempty.
By Lemma \ref{gridinterpretablelemma} there is a class of
supergrids $\mathcal{A}$ such that $\mf{G}\in \mathcal{A}$ and $\mathcal{A}$ is
uniformly first-order interpretable in the class of Cartesian frames with the domain $T$.
Therefore there exists a uniform interpretation $I'$ of $\mathcal{A}$ in the class of
Cartesian frames with the domain $T$.
Let $S$ be a finite nonempty set of tiles. Note that $S$ is by 
definition a set of proposition symbols $P_{t}$, where $t$ is a tile type.
Let $I$ be the $S$-expansion of the uniform interpretation $I'$ of $\mathcal{A}$ in the
class of Cartesian frames with the domain $T$.
Define 
$\psi_{S}\, :=\, \exists X\, \exists Y\, (\exists P_t)_{P_t\, \in\, S}
\bigl(\, \varphi_{\mi{Cf}}\, \wedge\, I(\, \varphi_{S}\, )\, \bigr).$
We will prove that for each
input $S$ to the tiling problem, we have $(T,\beta)\models \psi_{S}$ if and only if the
grid $\mf{G}$ is $S$-tilable.
Thereby we establish that there exists a computable reduction from the
\emph{complement problem} of the tiling problem
to the membership problem of the monadic $\Pi_1^1$-theory of $(T,\beta)$.
Since the tiling problem is $\Pi_1^0$-complete,
its complement problem is $\Sigma_1^0$-complete.\footnote{It is of course a
well-known triviality that the complement $\overline{A}$ of a problem $A$ is $\Sigma^0_1$-hard
if $A$ is $\Pi^0_1$-hard. 
Choose an arbitrary problem $B\in \Sigma^0_1$. 
By definition $\overline{B}\in\Pi^0_1$.
By the hardness of $A$, there is a computable 
reduction $f$ such that $x\in\overline{B}\Leftrightarrow f(x)\in A$, 
whence $x\in B\Leftrightarrow f(x)\in \overline{A}$.}
Let $S$ be an input to the tiling problem.
Assume first that there exists an $S$-tiling of the grid $\mf{G}$.
Therefore there exists an expansion $\mf{G}^*$ of the grid $\mf{G}$ to the
vocabulary $\{H,V\}\, \cup\, S$
such that $\mf{G}^*\models\varphi_{S}$.
Hence, by Lemma \ref{uniforminterpretationlemma} and since $\mf{G}\in\mathcal{A}$, there
exists a Cartesian frame $\mf{C}$ with the domain $T$ such that for some
expansion $\mf{C}^*$ of $\mf{C}$ to the vocabulary $\{\beta,X,Y\}\, \cup\, S$,
we have $\mf{C}^*\models I(\varphi_{S})$. On the
other hand, since $\mf{C}$ is a Cartesian frame, we have $\mf{C}^*\models \varphi_{\mi{Cf}}$.
Therefore $\mf{C}^*\models \varphi_{\mi{Cf}}\wedge I(\varphi_{S})$, and hence $(T,\beta)\models \psi_{S}$.
For the converse, assume that $(T,\beta)\models\psi_{S}$.
Therefore there exists an expansion $\mf{B}^*$ of $(T,\beta)$ to the
vocabulary $\{\beta,X,Y\}\, \cup\, S$
such that we have $\mf{B}^*\models \varphi_{\mi{Cf}}\, \wedge\, I(\varphi_{S})$. 
Since $\mf{B}^*\models \varphi_{\mi{Cf}}$, the $\{\beta,X,Y\}$-reduct of $\mf{B}^*$ is a Cartesian frame with the domain $T$. Therefore,
we conclude by Lemma \ref{uniforminterpretationlemma}
that $\mf{A}^*\models\varphi_{S}$ for some expansion $\mf{A}^*$ of some supergrid $\mf{A}\in\mathcal{A}$ to the
vocabulary $\{H,V\}\, \cup\, S$. Thus there exists a supergrid that $S$-tilable.
Hence the grid $\mf{G}$ is $S$-tilable.
\end{proof}
\begin{corollary}
Let $T\subseteq \rsq$ be such that $T$ extends linearly in $2D$. Let $\mathcal{C}$ be the class of
expansions $(T,\beta,(P_i)_{i\in\mathbb{N}})$ of $(T,\beta)$ with arbitrary unary predicates. The first-order theory of $\mathcal{C}$ is undecidable.
\end{corollary}
We note that $T$ extending linearly in $1D$ is not a sufficient condition for undecidability of the monadic $\Pi_1^1$-theory of $(T,\beta)$. The monadic $\Pi_1^1$-theory of $(\mathbb{R},\beta)$ is decidable; this follows trivially from the known result that the monadic $\Pi_1^1$-theory $(\mathbb{R},\leq)$ is decidable, see \cite{Burgess}. Also the monadic $\Pi_1^1$-theory of $(\mathbb{Q},\beta)$ is decidable since the $\mathrm{MSO}$ theory of $(\mathbb{Q},\leq)$ is decidable \cite{Rabin:1969}.
\begin{theorem}
Let $n\geq 2$ be an integer. The monadic $\Pi_1^1$-theory of the structure $\rp$ is $\Pi_1^1$-hard.
\end{theorem}
\begin{proof}
The proof is essentially the same as the proof of Theorem \ref{infinitelemma}. The main difference is that we use Lemma \ref{recurrencegridinterpr}
and interpret the recurrence grid $\mf{R}$ instead of a class of supergrids and hence obtain a reduction from the recurring tiling problem
instead of the ordinary tiling problem. Thereby we establish $\Pi^1_1$-hardness instead of $\Sigma_1^0$-hardness.
Due to the recurrence condition of the recurrent tiling problem, the result of Lemma \ref{nonstandardmodelsexcluded} that
there is an isomorphism from $(\mathbb{N},\mi{succ})$ to $(P,E)$---rather than an embedding---is essential.
\end{proof}
\begin{corollary}
Let $n\geq 2$ be an integer. Let $\mathcal{C}$ be the class of expansions $(\rpm,(P_i)_{i\in\mathbb{N}})$ of $\rp$
with arbitrary unary predicates. The first-order theory of $\mathcal{C}$ is not on any
level of the arithmetical hierarchy.
\end{corollary}

\section{Geometric structures $(T,\beta)$ with an undecidable weak monadic $\Pi_1^1$-theory}\label{section5}

In this section we prove that the
weak universal monadic second-order theory of any structure $(T,\beta)$ such that $T$ extends linearly in $2D$ is undecidable.
In fact, we show that any such theory is $\Pi^0_1$-hard. We
establish this by a reduction from the periodic tiling problem to the problem of deciding truth of $\{\beta\}$-sentences of
weak monadic $\Sigma^1_1$ in $(T,\beta)$.
The argument is based on interpreting tori in $(T,\beta)$.
Most notions used in this section are inherited either directly or with minor modification from Section \ref{main}.

Let $Q$ be a subset of  $T\subseteq\rsq$. We say that $Q$ is a \emph{finite sequence} in $T$ if $Q$ is a finite
nonempty set and the points in $Q$ are all collinear.

\begin{definition}
Let $T\subseteq\rsq$ and let $\beta$ be the corresponding betweenness relation.
 Let $P$ and $Q$ be finite sequences in $T$ such that the following
conditions hold.
\begin{enumerate}
\item
$P\cap Q=\{a_0\}$, where $a_0$ is a zero point of both $P$ and $Q$.
\item $P$ and $Q$ are non-singleton sequences.
\item There exists lines $L_P,L_Q$ in $T$ such that $L_P\neq L_Q$, $P\subseteq L_P$ and $Q\subseteq L_Q$.
\end{enumerate}
We call the structure $(T,\beta,P,Q)$ \emph{a finite Cartesian frame} with the domain $T$. The unique intersection point of $P$ and $Q$ is called the \emph{origo} of the
frame. If $\lvert P\rvert=m+1$ and $\lvert Q\rvert=n+1$,
we call $(T,\beta,P,Q)$ an \emph{$m\times n$ Cartesian frame} with the domain $T$.
\end{definition}

\begin{lemma}\label{finitedefinable}
Let $T\subseteq \rsq$, $n\geq 2$.  Let $\mathcal{C}$ be the class of all expansions $(T,\beta,P,Q)$ of $(T,\beta)$ by finite
unary relations $P$ and $Q$. The class of finite
Cartesian frames with the domain $T$ is definable with respect to $\mathcal{C}$ by a first-order sentence.
\end{lemma}
\begin{proof}
Straightforward.
\end{proof}

\begin{lemma}\label{torusinterpretablelemma}
Let $T\subseteq\mathbb{R}^n$, $n\geq 2$. Assume that $T$ extends linearly in $2D$.
The class of tori is uniformly first-order interpretable in the class of finite Cartesian frames with the domain $T$.
\end{lemma}
\begin{proof}
Let $\mf{C}=(T,\beta,P,Q)$ be a finite Cartesian frame.
We denote by $p_e\in P$ and $q_e\in Q$ the limit points of $P$ and $Q$, respectively. Clearly $p_e$ and $q_e$ are definable by a first-order formula with one free variable.

Define $\varphi^{\mi{fin}}_{\mi{Dom}}(u)\, :=\, \varphi_{\mi{Dom}}(u)$. Also
define the following variants of the $\{\beta,P,Q\}$-formulas $\varphi_H(u,v)$ and $\varphi_V(u,v)$
definined in Lemma \ref{gridinterpretablelemma}. Let
\begin{align*}
\varphi^{\mi{fin}}_H&(u,v):=\varphi_H(u,v)
\vee\Big(Qv\wedge\beta(v,u,p_e)\wedge
\forall x\big(\beta^*(u,x,p_e)\rightarrow\neg\varphi_{\mi{Dom}}^{\mi{fin}}(x)\big)\Big),\\
\varphi^{\mi{fin}}_V&(u,v):=\varphi_V(u,v)
\vee\Big(Pv\wedge\beta(v,u,q_e)\wedge
\forall x\big(\beta^*(u,x,q_e)\rightarrow\neg\varphi_{\mi{Dom}}^{\mi{fin}}(x)\big)\Big).
\end{align*}
Let $F_\mf{C}:=\{r\in T \mid \mf{C}\models \varphi^{\mi{fin}}_{Dom}(r)\}$. Define the structure $\mf{F}_\mf{C}=(F_\mf{C},H^{\mf{F}_\mf{C}},V^{\mf{T}_\mf{C}})$, where
\begin{align*}
H^{\mf{T}_\mf{C}}&:=\{(s,t)\in F_\mf{C}\times F_\mf{C}\mid \mf{C}\models \varphi^{\mi{fin}}_H(s,t)\}\text{ and}\\
V^{\mf{T}_\mf{C}}&:=\{(s,t)\in F_\mf{C}\times F_\mf{C}\mid \mf{C}\models \varphi^{\mi{fin}}_V(s,t)\}.
\end{align*}
It is straightforward to check that if $\mf{C}$ is an $m\times n$ Cartesian frame, then there exists a bijection $f$ from the domain of the $m\times n$ torus $\mf{T}_{m,n}=(T_{m,n},H_{m,n},V_{m,n})$ to $F_\mf{C}$ such that the following conditions hold for all $u,v\in T_{m,n}$.
\begin{enumerate}
\item $(u,v)\in H_{m,n}\Leftrightarrow \varphi^{\mi{fin}}_H(f(u),f(v))$,
\vspace{1pt}
\item $(u,v)\in V_{m,n}\Leftrightarrow \varphi^{\mi{fin}}_V(f(u),f(v))$.
\end{enumerate}
Notice that since $T$ extends linearly in $2D$, there exist finite Cartesian frames of all sizes in the class of finite Cartesian frames with the domain $T$. Hence the class of finite tori is uniformly first-order interpretable in the class of finite Cartesian frames with the domain $T$.
\end{proof}
\begin{theorem}
Let $T\subseteq \rsq$ and let $\beta$ be the corresponding betweenness relation.
Assume that $T$ extends linearly in $2D$. The
weak monadic $\Pi_1^1$-theory of $(T,\beta)$ is $\Pi^0_1$-hard.
\end{theorem}
\begin{proof}
Since $T$ extends linearly in $2D$, we have $n\geq 2$.
Let $\sigma=\{H,V\}$ be the vocabulary of tori, and let $\tau=\{\beta,X,Y\}$ be the vocabulary of finite Cartesian frames.
Let $C=\{\, (T,\beta,X,Y)\mid X\text{ and }Y \text{ are finite sets, } X,Y\subseteq T\, \}$.
By Lemma \ref{finitedefinable},
there exists a first-order $\tau$-sentence that
defines the class of finite Cartesian frames with the domain $T$ with respect to the
class $C$. Let $\varphi_{\mi{fcf}}$ denote such a sentence.
By Lemma \ref{tilingdefinablelemma}, every input $S$ to the periodic tiling
problem can be effectively associated
with a first-order $\sigma\cup S$-sentence $\varphi_S$ such that
for all tori $\mf{B}$, the torus $\mf{B}$ is $S$-tilable iff there is an
expansion $\mf{B}^*$ of $\mf{B}$ to the
vocabulary $\sigma\cup S$
such that $\mf{B}^*\models\varphi_{S}$.
By Lemma \ref{torusinterpretablelemma}, the class of tori is uniformly first-order interpretable in the class of
finite Cartesian frames with the domain $T$.
Let $S$ be a finite nonempty set of tiles and let $J$ be the $S$-expansion of the uniform interpretation of the class of tori in the class of finite Cartesian frames
 with the domain $T$. Let $\phi_S$ denote the following monadic $\Sigma^1_1$-sentence.
\[
\exists X\exists Y(\exists P_t)_{P_t\in S}(\varphi_{\mi{fcf}}\wedge J(\varphi_S)).
\]
We will show that $(T,\beta)\models \phi_S$ if and only if there exists an $S$-tilable torus $\mf{D}$.
First assume that there is an $S$-tilable torus $\mf{D}$.
Therefore, by Lemma \ref{tilingdefinablelemma}, there is an expansion $\mf{D}^*$ of $\mf{D}$ to the vocabulary $\sigma\cup S$ such that
$\mf{D}^*\models \varphi_S$. Since the class of tori is $J$-interpretable in the class of finite Cartesian frames with the domain $T$ and $\mf{D}^*\models
\varphi_S$, it follows by Lemma \ref{uniforminterpretationlemma} that there is a finite Cartesian frame $\mf{C}$ with the domain $T$ and an expansion
$\mf{C}^*$ of $\mf{C}$ to the vocabulary $\tau\cup T$ such that $\mf{C}^*\models J(\varphi_S)$.
Therefore $\mf{C}\models (\exists P_t)_{P_t\in S} J(\varphi_S)$. Since there exists a finite Cartesian frame with 
the domain $T$ that satisfies $(\exists P_t)_{P_t\in S}J(\varphi_S)$, we can conclude that
\[
(T,\beta)\models \exists X\exists Y(\exists P_t)_{P_t\in S}(\varphi_{\mi{fcf}}\wedge J(\varphi_S)).
\]
If, on the other hand, it holds that
\[
(T,\beta)\models \exists X\exists Y(\exists P_t)_{P_t\in S}((\varphi_{\mi{fcf}}\wedge J(\varphi_S)),
\]
it follows that there is a finite Cartesian frame $\mf{C}$ with the domain $T$ such that $\mf{C}\models(\exists P_t)_{P_t\in S} J(\varphi_S)$.
Therefore there exists an expansion $\mf{C}^*$ of $\mf{C}$ to the vocabulary $\tau\cup T$ such that $\mf{C}^*\models J(\varphi_T)$. Since the
class of tori is uniformly $J$-interpretable in the class of finite Cartesian frames with
the domain $T$ and $\mf{C}^*\models J(\varphi_S)$, there is by Lemma
\ref{uniforminterpretationlemma} an expansion $\mf{D}^*$ of a torus $\mf{D}$ to the vocabulary $\sigma\cup S$ such
that $\mf{D}^*\models \varphi_S$. Now by Lemma \ref{tilingdefinablelemma}, $\mf{D}$ is $S$-tilable.
Hence there is a torus which is $S$-tilable.
We have now shown that for any finite set of tiles $S$ it holds that there is a torus which is $S$-tilable
if and only if $(T,\beta)\models \phi_S$. Hence we have
reduced the periodic tiling problem to the problem of deciding truth of $\{\beta\}$-sentences of weak monadic $\Sigma^1_1$ in $(T,\beta)$. From the
$\Sigma^0_1$-completeness of the periodic tiling problem (Theorem \ref{periodictilingcomplete}), we conclude that
the weak monadic $\Sigma_1^1$-theory of the
structure $(T,\beta)$ is $\Sigma^0_1$-hard. Therefore the membership problem of the weak monadic $\Pi_1^1$-theory of the
structure $(T,\beta)$, is $\Pi^0_1$-hard.
\end{proof}
\begin{corollary}
Let $T\subseteq \rsq$ be a set such that $T$ extends linearly in $2D$. Let $\mathcal{C}$ be the class of
expansions $(T,\beta,(P_i)_{i\in\mathbb{N}})$ of $(T,\beta)$ with finite unary
predicates. The first-order theory of $\mathcal{C}$ is undecidable.
\end{corollary}
\section{Conclusions}
We have studied first-order theories of geometric structures $(T,\beta)$, $T\subseteq\rsq$, expanded with (finite) unary predicates.
We have established that for $n\geq 2$, the first-order theory of the class of all expansions of $\rp$ with arbitrary unary predicates is
highly undecidable ($\Pi_1^1$-hard). This refutes a conjecture from the article \cite{vanBenthemAiello:2002} of Aiello and van Benthem.
In addition, we have established the following for any geometric structure $(T,\beta)$ that extends linearly in $2D$.
\begin{enumerate}
\item
The first-order theory of the class of expansions of $(T,\beta)$ with arbitary unary predicates is $\Sigma_1^0$-hard.
\item
The first-order theory of the class of expansions of $(T,\beta)$ with finite unary predicates is $\Pi_1^0$-hard.
\end{enumerate}
Geometric structures that extend linearly in $2D$ include, for example, the rational plane $(\mathbb{Q}^2,\beta)$ and the real unit
rectangle $([0,1]^2,\beta)$, to name a few.
The techniques used in the proofs can be easily modified to yield undecidability of first-order theories of a significant variety of natural
restricted expansion classes of the affine real plane $(\mathbb{R}^2,\beta)$, such as those with unary predicates denoting polygons, finite unions of closed rectangles, and real algebraic sets, for example. Such classes could be interesting from the point of view of applications.

In addition to studying issues of decidability, we briefly compared the expressivities of universal monadic second-order logic and weak universal monadic second-order logic. While the two are incomparable in general, we established that over any class of expansions of $\rp$, it is no longer the case. We showed that finiteness of a unary predicate is definable by a first-order sentence, and hence obtained translations from $\forall\mathrm{WMSO}$ into $\forall\mathrm{MSO}$ and from $\mathrm{WMSO}$ into $\mathrm{MSO}$.

Our original objective to study weak monadic second order logic over $\rp$ was to identify decidable logics of space with distinguished regions.
Due to the ubiquitous applicability of the tiling methods, this pursuit gave way to identifying several undecidable theories of geometry.
Hence we shall look elsewhere in order to identify well behaved natural decidable logics of space. Possible interesting directions include considering natural fragments of first-order logic over expansions of $\rp$, and also other geometries.
Related results could provide insight, for example, in the background theory of modal spatial logics.





\end{document}